\documentclass[11pt]{amsart}
\usepackage{amsmath, amssymb}
\usepackage{amscd}
\usepackage{verbatim}
\usepackage{amssymb,latexsym,amsmath,amsfonts}
\usepackage{latexsym}
\usepackage[mathscr]{eucal}
\usepackage{bm}

\title[Decomposition of a tetrablock contraction]
{Canonical decomposition of a tetrablock contraction and operator
model}

\author{Sourav Pal}
\address{Department of Mathematics, Indian Institute of Technology Bombay, Mumbai - 400076, India.}
\email{sourav@math.iitb.ac.in,
souravmaths@gmail.com}
\thanks{The author was supported by the INSPIRE Faculty Award
(Award No. DST/INSPIRE/04/2014/001462) of DST, India.}



\def ~{\hspace{1mm}}

\def\textmatrix#1&#2\\#3&#4\\{\bigl({#1 \atop #3}\ {#2 \atop #4}\bigr)}
\def\dispmatrix#1&#2\\#3&#4\\{\left({#1 \atop #3}\ {#2 \atop #4}\right)}
\newcommand{\beg}{\begin{equation}}
\newcommand{\eeg}{\end{equation}}
\newcommand{\ben}{\begin{eqnarray*}}
\newcommand{\een}{\end{eqnarray*}}

\newtheorem{thm}{Theorem}[section]

\newtheorem{lem}[thm]{Lemma}

\newtheorem{prop}[thm]{Proposition}
\numberwithin{equation}{section}
\theoremstyle{definition}
\newtheorem{defn}[thm]{Definition}
\newtheorem{rem}[thm]{Remark}

\def\textmatrix#1&#2\\#3&#4\\{\bigl({#1 \atop #3}\ {#2 \atop #4}\bigr)}
\def\dispmatrix#1&#2\\#3&#4\\{\left({#1 \atop #3}\ {#2 \atop #4}\right)}

\begin{document}

\begin{abstract}
A triple of commuting operators for which the closed tetrablock
$\overline{\mathbb E}$ is a spectral set is called a tetrablock
contraction or an $\mathbb E$-contraction. The set $\mathbb E$ is
defined as
\[
\mathbb E = \{ (x_1,x_2,x_3)\in\mathbb C^3\,:\,
1-zx_1-wx_2+zwx_3\neq 0 \textup{ whenever } |z|\leq 1, |w|\leq 1
\}.
\]
We show that every $\mathbb E$-contraction can be uniquely written
as a direct sum of an $\mathbb E$-unitary and a completely
non-unitary $\mathbb E$-contraction. It is analogous to the
canonical decomposition of a contraction operator into a unitary
and a completely non-unitary contraction. We produce a concrete
operator model for such a triple satisfying some conditions.

\end{abstract}

\maketitle

\section{Introduction}

A compact subset $X$ of $\mathbb C^n$ is said to be a
\textit{spectral set} for a commuting $n$-tuple of bounded
operators $\underline{T}=(T_1,\hdots,T_n)$ defined on a Hilbert
space $\mathcal H$ if the Taylor joint spectrum
$\sigma(\underline{T})$ of $\underline{T}$ is a subset of $X$ and
\[
\|r(\underline{T})\|\leq \|r\|_{\infty,
X}=\sup\{|r(z_1,\hdots,z_n)|\,:\,(z_1,\hdots,z_n)\in\ X\}\,,
\]
for all rational functions $r$ in $\mathcal R(X)$. Here $\mathcal
R(X)$ denotes the algebra of all rational functions on $X$, that
is, all quotients $p/q$ of holomorphic polynomials $p,q$ in
$n$-variables for which $q$ has no zeros in $X$. A triple of
commuting operators $(A,B,P)$ for which the closure of the
tetrablock ${\mathbb E}$, where
\[
\mathbb E = \{ (x_1,x_2,x_3)\in\mathbb C^3\,:\,
1-zx_1-wx_2+zwx_3\neq 0 \textup{ whenever } |z|\leq 1, |w|\leq 1
\},
\]
is a spectral set is called a \textit{tetrablock contraction} or
an $\mathbb E$-\textit{contraction}.

Complex geometry, function theory and operator theory on the
tetrablock have been widely studied by a number of mathematicians
\cite{awy,awy-cor,tirtha,tirtha-sau,EZ,EKZ,sourav2,young,Zwo} over
past one decade because of the relevance of this domain to
$\mu$-synthesis problem and $H^{\infty}$ control theory. The
following result from \cite{awy} (Theorem 2.4 in \cite{awy})
characterizes points in $\mathbb E$ and $\overline{\mathbb E}$ and
provides a geometric description of the tetrablock.
\begin{thm}\label{thm:pre1}
A point $(x_1,x_2,x_3)\in \mathbb C^3$ is in $\overline{\mathbb
E}$ if and only if $|x_3|\leq 1$ and there exist $c_1,c_2 \in
\mathbb C$ such that $|c_1|+|c_2| \leq 1$ and $x_1=c_1 +
\bar{c_2}x_3, \quad x_2=c_2+\bar{c_1}x_3$.
\end{thm}
It is clear from the above result that the closed tetrablock
$\overline{\mathbb E}$ lives inside the closed tridisc
$\overline{\mathbb D^3}$ and consequently an $\mathbb
E$-contraction consists of commuting contractions. It is evident
from the definition that if $(A,B,P)$ is an $\mathbb
E$-contraction then so is its adjoint $(A^*,B^*,P^*)$. We briefly
recall from literature some special classes of $\mathbb
E$-contractions which are analogous to unitaries, isometries,
co-isometries etc. in one variable operator theory.
\begin{defn}
Let $A,B,P$ be commuting operators on a Hilbert space $\mathcal
H$. We say that $(A,B,P)$ is
\begin{itemize}
\item [(i)] an $\mathbb E$-\textit{unitary} if $A,B,P$ are normal
operators and the joint spectrum $\sigma(A,B,P)$ is contained in
the distinguished boundary $b\overline{\mathbb E}$ of the
tetrablock , where
\begin{align*}
b \overline{\mathbb E} &=\{(x_1,x_2,x_3)\in \mathbb C^3\,:\,
x_1=\bar{x_2}x_3,|x_2|\leq 1, |x_3|=1 \}\\& =\{ (x_1,x_2,x_3)\in
\overline{\mathbb E}\,:\, |x_3|=1 \}.
\end{align*}

\item [(ii)] an $\mathbb E$-\textit{isometry} if there exists a
Hilbert space $\mathcal K$ containing $\mathcal H$ and an $\mathbb
E$-unitary $(\tilde{A},\tilde{B},\tilde{P})$ on $\mathcal K$ such
that $\mathcal H$ is a common invariant subspace of $A,B,P$ and
that $A=\tilde{A}|_{\mathcal H}, B=\tilde{B}|_{\mathcal H},
P=\tilde{P}|_{\mathcal H}$ ; \item [(iii)] an $\mathbb
E$-\textit{co-isometry} if $(A^*,B^*,P^*)$ is an $\mathbb
E$-isometry ; \item [(iv)] a \textit{completely non-unitary}
$\mathbb E$-contraction if $(A,B.P)$ is an $\mathbb E$-contraction
and $P$ is a completely non-unitary contraction ; \item [(v)] a
\textit{pure} $\mathbb E$-contraction if $(A,B.P)$ is an $\mathbb
E$-contraction and $P$ is a pure contraction, that is,
$P^{*n}\rightarrow 0$ strongly as $n\rightarrow \infty $.
\end{itemize}
\end{defn}

\begin{defn}
Let $(A,B,P)$ be an $\mathbb E$-contraction on a Hilbert space
$\mathcal H$. A commuting triple $(V_1,V_2,V_3)$ on $\mathcal K$
is said to be an $\mathbb E$-isometric dilation of $(A,B,P)$ if
$(V_1,V_2,V_3)$ is an $\mathbb E$-isometry, $\mathcal H \subseteq
\mathcal K$ and
\[
f(A,B,P)=P_{\mathcal H}f(V_1,V_2,V_3)|_{\mathcal H}
\]
for every holomorphic polynomial $f$ in three variables. Here
$P_{\mathcal H}$ denotes the projection onto $\mathcal H$.
Moreover, this dilation is called minimal if
\[
\mathcal K =\overline{\text{span}} \{ f(V_1,V_2,V_3)h\,:\,
h\in\mathcal H \,, f\in \mathbb C[z_1,z_2,z_3] \}.
\]
\end{defn}

It was a path breaking discovery by von Neumann, \cite{vN}, that a
bounded operator $T$ is a contraction if and only if the closed
unit disc $\overline{\mathbb D}$ in the complex plane is a
spectral set for $T$. It is well known that to every contraction
$T$ on a Hilbert space $\mathcal H$ there corresponds a
decomposition of $\mathcal H$ into an orthogonal sum of two
subspaces reducing $T$, say $\mathcal H=\mathcal H_1\oplus\mathcal
H_2$ such that $T|_{\mathcal H_1}$ is unitary and $T|_{\mathcal
H_2}$ is completely non-unitary; $\mathcal H_1$ or $\mathcal H_2$
may equal the trivial subspace $\{0\}$. This decomposition is
uniquely determined and is called the canonical decomposition of a
contraction (see Theorem 3.2 in Ch-I, \cite{nagy} for details).
Indeed, $\mathcal H_1$ consists of those elements $h\in\mathcal H$
for which
\[
\|T^nh\|=\|h\|=\|T^{*n}h\| \quad \quad (n=1,2,\hdots) \,.
\]
The main aim of this article is to show that an $\mathbb
E$-contraction admits an analogous decomposition into an $\mathbb
E$-unitary and a completely non-unitary $\mathbb E$-contraction.
Indeed, in Theorem \ref{thm:decomp}, one of the main results of
this paper, we show that for an $\mathbb E$-contraction $(A,B,P)$
defined on $\mathcal H$ if $\mathcal H_1\oplus \mathcal H_2$ is
the unique orthogonal decomposition of $\mathcal H$ into reducing
subspaces of $P$ such that $P|_{\mathcal H_1}$ is a unitary and
$P|_{\mathcal H_2}$ is a completely non-unitary, then $\mathcal
H_1, \mathcal H_2$ also reduce $A,B$ ; $(A|_{\mathcal
H_1},B|_{\mathcal H_1},P|_{\mathcal H_1})$ is an $\mathbb
E$-unitary and $(A|_{\mathcal H_2},B|_{\mathcal H_2},P|_{\mathcal
H_2})$ is a completely non-unitary $\mathbb E$-contraction.\\

The other contribution of this article is that we produce a
concrete operator model for an $\mathbb E$-contraction which
satisfies some conditions. Before getting into the details of it
we recall a few words from the literature about the fundamental
equations and the fundamental operators related to an $\mathbb
E$-contraction.

For an $\mathbb E$-contraction $(A,B,P)$, the \textit{fundamental
equations} were defined in \cite{tirtha} as
\begin{equation}\label{eqn:funda}
A-B^*P=D_PX_1D_P\,,\quad B-A^*P=D_PX_2D_P\,; \quad \quad
D_{P}=(I-P^*P)^{\frac{1}{2}}.
\end{equation}
It was proved in \cite{tirtha} (Theorem 3.5, \cite{tirtha}) that
corresponding to every $\mathbb E$-contraction $(A,B,P)$ there
were two unique operators $F_1,F_2$ in $\mathcal B(\mathcal
D_{P})$ that satisfied the fundamental equations, i.e,
\[
A-B^*P=D_{P}F_1D_{P}\, , \; B-A^*P=D_{P}F_2D_{P}\,.
\]
Here $\mathcal D_{P}=\overline{Ran}\,D_{P}$ and is called the
defect space of $P$. Also $\mathcal B(\mathcal H)$, for a Hilbert
space $\mathcal H$, always denotes the algebra of bounded
operators on $\mathcal H$. An explicit $\mathbb E$-isometric
dilation was constructed for a particular class of $\mathbb
E$-contractions in \cite{tirtha} (Theorem 6.1, \cite{tirtha}) and
$F_1,F_2$ played the fundamental role in that explicit
construction of dilation. For their pivotal role in the dilation,
$F_1$ and $F_2$ were called the \textit{fundamental
operators} of $(A,B,P)$.\\

It was shown in \cite{tirtha} (Theorem 6.1, \cite{tirtha}) that an
$\mathbb E$-contraction $(A,B,P)$ dilated to an $\mathbb
E$-isometry if the corresponding fundamental operators $F_1,F_2$
satisfied $[F_1,F_2]=0$ and $[F_1^*,F_1]=[F_2^*,F_2]$. Here
$[S_1,S_2]=S_1S_2-S_2S_1$ for any two bounded operators $S_1,S_2$.
On the other hand there are $\mathbb E$-contractions which do not
dilate. Indeed, an $\mathbb E$-contraction may not dilate to an
$\mathbb E$-isometry if $[F_1^*,F_1]\neq [F_2^*,F_2]$; it has been
established in \cite{sourav1} by a counter example. So it turns
out that those two conditions are very crucial for an $\mathbb
E$-contraction. In Theorem \ref{thm:model2}, we construct a
concrete model for an $\mathbb E$-contraction $(A,B,P)$ when the
fundamental operators $F_{1*},F_{2*}$ of $(A^*,B^*,P^*)$ satisfy
$[F_{1*},F_{2*}]=0$ and $[F_{1*}^*,F_{1*}]=[F_{2*}^*,F_{2*}]$. In
brief, such an $\mathbb E$-contraction is the restriction to a
common invariant subspace of an $\mathbb E$-co-isometry and every
$\mathbb E$-co-isometry is expressible as the orthogonal direct
sum of an $\mathbb E$-unitary and a pure $\mathbb E$-co-isometry,
which has a model on the vectorial Hardy space $H^2(\mathcal
D_{T_3})$, where $T_3^*$ is
the minimal isometric dilation of $P^*$.\\

In section 2, we accumulate a few new results about $\mathbb
E$-contractions and also state some results from the literature
which will be used in sequel.

\section{The set $\mathbb E$ and $\mathbb E$-contractions}

We begin this section with a lemma that characterizes the points
in $\overline{\mathbb E}$.

\begin{lem}\label{lem:1}
$(x_1,x_2,x_3)\in\overline{\mathbb E}$ if and only if $(\omega
x_1,\omega x_2, \omega^2 x_3)\in \overline{\mathbb E}$ for all
$\omega\in\mathbb T$.
\end{lem}
\begin{proof}
Let $(x_1,x_2,x_3)\in\overline{\mathbb E}$. Then by Theorem
\ref{thm:pre1}, $|x_3|\leq 1$ and there are complex numbers
$c_1,c_2$ with $|c_1|+|c_2|\leq 1$ such that
$x_1=c_1+\bar{c_2}x_3,\quad x_2=c_2+\bar{c_1}x_3$. For
$\omega\in\mathbb T$ if we choose $d_1=\omega c_1 \text{ and }
d_2=\omega c_2$ we see that $|d_1|+|d_2|\leq 1$ and
\begin{align*}
& \omega x_1 =\omega(c_1+\bar{c_2}x_3)=\omega c_1+\overline{\omega
c_2}(\omega^2 x_3)=d_1+\bar{d_2}(\omega^2 x_3)\,, \\& \omega
x_2=\omega(c_2+\bar{c_1}x_3)=\omega c_2+\overline{\omega
c_1}(\omega^2 x_3)=d_2+\bar{d_1}(\omega^2 x_3).
\end{align*}
Therefore, by Theorem \ref{thm:pre1}, $(\omega x_1,\omega x_2,
\omega^2 x_3)\in \mathbb E$. The other side of the proof is
trivial.

\end{proof}

The following lemma simplifies the definition of $\mathbb
E$-contraction.

\begin{lem}\label{lem:pre1}
A triple of commuting operators $(A,B,P)$ is an $\mathbb
E$-contraction if and only if
\[
\|f(A,B,P)\|\leq \|f\|_{\infty,\overline{\mathbb E}}=\sup
\{|f(x_1,x_2,x_3)|\,:\,(x_1,x_2,x_3)\in\overline{\mathbb E}\}
\]
for all holomorphic polynomials $f$ in three variables.
\end{lem}

This actually follows from the fact that $\overline{\mathbb E}$ is
polynomially convex. A proof to this could be found in
\cite{tirtha} (Lemma 3.3, \cite{tirtha}).

\begin{lem}\label{lem:2}
Let $(A,B,P)$ be an $\mathbb E$-contraction. Then so is $(\omega
A,\omega B,\omega^2 P)$ for any $\omega\in\mathbb T$.
\end{lem}
\begin{proof}
Let $f(x_1,x_2,x_3)$ be a holomorphic polynomial in the
co-ordinates of $\overline{\mathbb E}$ and for $\omega\in\mathbb
T$ let $ f_1(x_1,x_2,x_3)=f(\omega x_1,\omega x_2,\omega^2 x_3). $
It is evident from Lemma \ref{lem:1} that
\[
\sup\{|f(x_1,x_2,x_3)|\,:\,(x_1,x_2,x_3)\in\overline{\mathbb
E}\}=\sup\{|f_1(x_1,x_2,x_3)|\,:\,(x_1,x_2,x_3)\in\overline{\mathbb
E}\}.
\]
Therefore,
\begin{align*}
\|f(\omega A, \omega B, \omega^2 P)\|& =\|f_1(A,B,P)\| \\& \leq
\|f_1\|_{\infty, \overline{\mathbb E}} \\&
=\|f\|_{\infty,\overline{\mathbb E}}.
\end{align*}
Therefore, by Lemma \ref{lem:pre1}, $(\omega A, \omega B, \omega^2
P)$ is an $\mathbb E$-contraction.
\end{proof}

The following result was proved in \cite{tirtha} (see Theorem 3.5
in \cite{tirtha}).

\begin{thm}\label{thm:pre2}
Let $(A,B,P)$ be an $\mathbb E$-contraction. Then the operator
functions $\rho_1$ and $\rho_2$ defined by

\begin{align*}
\rho_1(A,B,P)&=(I-P^*P)+(A^*A-B^*B)-2\text{ Re
}(A-B^*P)\,,\\
\rho_2(A,B,P)&=(I-P^*P)+(B^*B-A^*A)-2\text{ Re }(B-A^*P)
\end{align*}
satisfy
\[
\rho_1(A,zB,zP)\geq 0 \text{ and } \rho_2(A,zB,zP)\geq 0 \text{
for all } z\in\overline{\mathbb D}.
\]

\end{thm}

\begin{lem}\label{lem:3}
Let $(A,B,P)$ be an $\mathbb E$-contraction. Then for $i=1,2$,
$\rho_i(\omega A, \omega B, \omega^2 P)\geq 0$ for all
$\omega\in\mathbb T$.
\end{lem}
\begin{proof}
By Theorem \ref{thm:pre2},
\[
\rho_1(A,B,P)\geq 0 \text{ and } \rho_2(A,B,P)\geq 0.
\]
Since $(\omega A,\omega B,\omega^2 P)$ is an $\mathbb
E$-contraction for every $\omega$ in $\mathbb T$ by Lemma
\ref{lem:2}, we have that
\[
\rho_1(\omega A,\omega B,\omega^2 P)\geq 0 \text{ and }
\rho_2(\omega A,\omega B,\omega^2 P)\geq 0 \,.
\]
\end{proof}

The following theorem provides a set of characterizations for
$\mathbb E$-unitaries and for a proof to this one can see Theorem
5.4 in \cite{tirtha}.

\begin{thm}\label{thm:tu}
Let $\underline N = (N_1, N_2, N_3)$ be a commuting triple of
bounded operators. Then the following are equivalent.

\begin{enumerate}

\item $\underline N$ is an $\mathbb E$-unitary,

\item $N_3$ is a unitary and $\underline N$ is an $\mathbb
E$-contraction,

\item $N_3$ is a unitary, $N_2$ is a contraction and $N_1 = N_2^*
N_3$.
\end{enumerate}
\end{thm}

\noindent Here is a structure theorem for the $\mathbb
E$-isometries (see Theorem 5.6 and 5.7 in \cite{tirtha}).

\begin{thm} \label{thm:ti}

Let $\underline V = (V_1, V_2, V_3)$ be a commuting triple of
bounded operators. Then the following are equivalent.

\begin{enumerate}

\item $\underline V$ is an $\mathbb E$-isometry.

\item $V_3$ is an isometry and $\underline V$ is an $\mathbb
E$-contraction.

\item $V_3$ is an isometry, $V_2$ is a contraction and $V_1=V_2^*
V_3$.

\item (\textit{Wold decomposition}) $\mathcal H$ has a
decomposition $\mathcal H=\mathcal H_1\oplus \mathcal H_2$ into
reducing subspaces of $V_1,V_2,V_3$ such that $(V_1|_{\mathcal
H_1},V_2|_{\mathcal H_1},V_3|_{\mathcal H_1})$ is an $\mathbb
E$-unitary and $(V_1|_{\mathcal H_2},V_2|_{\mathcal
H_2},V_3|_{\mathcal H_2})$ is a pure $\mathbb E$-isometry.
\end{enumerate}
\end{thm}

\section{Canonical decomposition of an $\mathbb E$-contraction}

\begin{thm}\label{thm:decomp}
Let $(A,B,P)$ be an $\mathbb E$-contraction on a Hilbert space
$\mathcal H$. Let $\mathcal H_1$ be the maximal subspace of
$\mathcal H$ which reduces $P$ and on which $P$ is unitary. Let
$\mathcal H_2=\mathcal H\ominus \mathcal H_1$. Then $\mathcal
H_1,\mathcal H_2$ reduce $A, B$; $(A|_{\mathcal H_1},B|_{\mathcal
H_1},P|_{\mathcal H_1})$ is an $\mathbb E$-unitary and
$(A|_{\mathcal H_2},B|_{\mathcal H_2},P|_{\mathcal H_2})$ is a
completely non-unitary $\mathbb E$-contraction. The subspaces
$\mathcal H_1$ or $\mathcal H_2$ may equal to the trivial subspace
$\{0\}$.
\end{thm}

\begin{proof}
It is obvious that if $P$ is a completely non-unitary contraction
then $\mathcal H_1=\{0\}$ and if $P$ is a unitary then $\mathcal
H=\mathcal H_1$ and so $\mathcal H_2=\{0\}$. In such cases the
theorem is trivial. So let us suppose that $P$ is neither a
unitary nor a completely non unitary contraction. Let
\[
A=
\begin{bmatrix}
A_{11}&A_{12}\\
A_{21}&A_{22}
\end{bmatrix}\,,\,
B=
\begin{bmatrix}
B_{11}&B_{12}\\
B_{21}&B_{22}
\end{bmatrix}
\text{ and } P=
\begin{bmatrix}
P_1&0\\
0&P_2
\end{bmatrix}
\]
with respect to the decomposition $\mathcal H=\mathcal H_1\oplus
\mathcal H_2$, so that $P_1$ is a unitary and $P_2$ is completely
non-unitary. Since $P_2$ is completely non-unitary it follows that
if $x\in\mathcal H$ and
\[
\|P_2^nx\|=\|x\|=\|{P_2^*}^nx\|, \quad n=1,2,\hdots
\]
then $x=0$.

The fact that $A$ and $P$ commute tells us that

\begin{align}
A_{11}P_1&=P_1A_{11}    & A_{12}P_2=P_1A_{12}\,, \label{eqn:1} \\
A_{21}P_1&=P_2A_{21}    & A_{22}P_2=P_2A_{22}\,. \label{eqn:2}
\end{align}
Also by commutativity of $B$ and $P$ we have

\begin{align}
B_{11}P_1&=P_1B_{11}    & B_{12}P_2=P_1B_{12}\,, \label{eqn:3} \\
B_{21}P_1&=P_2B_{21}    & B_{22}P_2=P_2B_{22}\,. \label{eqn:4}
\end{align}
By Lemma \ref{lem:3}, we have for all $\omega, \beta\in\mathbb T$,
\begin{align*}
\rho_1(\omega A,\omega B,\omega^2 P)&=(I-P^*P)+(A^*A-B^*B)-2\text{
Re }\omega(A-B^*P)\geq 0 \,,\\
\rho_2(\beta A,\beta B, \beta^2 P)&=(I-P^*P)+(B^*B-A^*A)-2\text{
Re }\beta(B-A^*P) \geq 0 \,.
\end{align*}
Adding $\rho_1$ and $\rho_2$ we get
\[
(I-P^*P)-\text{Re }\omega(A-B^*P)-\text{Re } \beta(B-A^*P)\geq 0
\]
that is
\begin{align}\label{eqn:5}
\begin{bmatrix}
0&0\\
0&I-P_2^*P_2
\end{bmatrix}
-& \text{ Re }\omega
\begin{bmatrix}
A_{11}-B_{11}^*P_1&A_{12}-B_{21}^*P_2\\
A_{21}-B_{12}^*P_1&A_{22}-B_{22}^*P_2
\end{bmatrix} \\
-&\text{ Re }\beta
\begin{bmatrix}
B_{11}-A_{11}^*P_1&B_{12}-A_{21}^*P_2\\
B_{21}-A_{12}^*P_1&B_{22}-A_{22}^*P_2
\end{bmatrix}\, \geq 0 \notag
\end{align}
for all $\omega,\beta\in\mathbb T$. Since the matrix in the left
hand side of (\ref{eqn:5}) is self-adjoint, if we write
(\ref{eqn:5}) as

\begin{equation}\label{eqn:6}
\begin{bmatrix}
R&X\\
X^*&Q
\end{bmatrix}
\geq 0\,,
\end{equation}
then

\begin{eqnarray*}\begin{cases}
&(\mbox{i})\; R\,, Q \geq 0 \text{ and } R=-\text{ Re }\omega (
A_{11}-B_{11}^*P_1) -\text{ Re }\beta (B_{11}-A_{11}^*P_1)\\&
(\mbox{ii}) X= -\frac{1}{2} \{ \omega (
A_{12}-B_{21}^*P_2)+\bar{\omega}(A_{21}^*-P_1^*B_{12})\\& \quad
\quad \quad \quad + \beta
(B_{12}-A_{21}^*P_2)+\bar{\beta}(B_{21}^*-P_1^*A_{12}) \}
\\&(\mbox{iii})\; Q=(I-P_2^*P_2)-\text{ Re }\omega (A_{22}-B_{22}^*P_2) -
\text{ Re }\beta (B_{22}-A_{22}^*P_2) \;.
\end{cases}
\end{eqnarray*}

Since the left hand side of (\ref{eqn:6}) is a positive
semi-definite matrix for every $\omega$ and $\beta$, if we choose
$\beta=1$ and $\beta=-1$ respectively then consideration of the
$(1,1)$ block reveals that
\[
\omega(A_{11}-B_{11}^*P_1)+\bar{\omega}(A_{11}^*-P_1^*B_{11})\leq
0
\]
for all $\omega\in\mathbb T$. Choosing $\omega =\pm 1$ we get
\begin{equation}\label{eqn:7}
(A_{11}-B_{11}^*P_1)+(A_{11}^*-P_1^*B_{11})=0
\end{equation}
and choosing $\omega =\pm i$ we get
\begin{equation}\label{eqn:8}
(A_{11}-B_{11}^*P_1)-(A_{11}^*-P_1^*B_{11})=0\,.
\end{equation}
Therefore, from (\ref{eqn:7}) and (\ref{eqn:8}) we get
\[
A_{11}=B_{11}^*P_1\,,
\]
where $P_1$ is unitary. Similarly, we can show that
\[
B_{11}=A_{11}^*P_1\,.
\]
Therefore, $R=0$. Since $(A,B,P)$ is an $\mathbb E$-contraction,
$\|B\|\leq 1$ and hence $\|B_{11}\|\leq 1$ also. Therefore, by
part-(3) of Theorem \ref{thm:tu},
$(A_{11},B_{11},P_1)$ is an $\mathbb E$-unitary.\\

Now we apply Proposition 1.3.2 of \cite{bhatia} to the positive
semi-definite matrix in the left hand side of (\ref{eqn:6}). This
Proposition states that if $R,Q \geq 0$ then $\begin{bmatrix} R&X
\\ X^*&Q
\end{bmatrix} \geq 0$ if and only if $X=R^{1/2}KQ^{1/2}$ for
some contraction $K$.\\

\noindent Since $R=0$, we have $X=0$. Therefore,
\[
\omega ( A_{12}-B_{21}^*P_2)+\bar{\omega}(A_{21}^*-P_1^*B_{12})+
\beta (B_{12}-A_{21}^*P_2)+\bar{\beta}(B_{21}^*-P_1^*A_{12}) =0\;,
\]
for all $\omega,\beta \in\mathbb T$. Choosing $\beta =\pm 1$ we
get
\[
\omega (
A_{12}-B_{21}^*P_2)+\bar{\omega}(A_{21}^*-P_1^*B_{12})=0\;,
\]
for all $\omega \in \mathbb T$. With the choices $\omega=1,i$ ,
this gives
\[
A_{12}=B_{21}^*P_2\,.
\]
Therefore, we also have
\[
A_{21}^*=P_1^*B_{12}\,.
\]
Similarly, we can prove that
\[
B_{12}=A_{21}^*P_2\,,\quad B_{21}^*=P_1^*A_{12}\,.
\]
Thus, we have the following equations
\begin{align}
A_{12}&=B_{21}^*P_2         & A_{21}^*&=P_1^*B_{12} \label{eqn:9}\\
B_{12}&=A_{21}^*P_2         & B_{21}^*&=P_1^*A_{12}\,.
\label{eqn:10}
\end{align}
Thus from (\ref{eqn:9}), $A_{21}=B_{12}^*P_1$ and together with
the first equation in (\ref{eqn:2}), this implies that
\[
B_{12}^*P_1^2=A_{21}P_1=P_2A_{21}=P_2B_{12}^*P_1
\]
and hence
\begin{equation}\label{eqn:11}
B_{12}^*P_1=P_2B_{12}^*\,.
\end{equation}
From equations in (\ref{eqn:3}) and (\ref{eqn:11}) we have that
\[
B_{12}P_2=P_1B_{12}\,, \quad B_{12}{P_2^*}={P_1^*}B_{12}.
\]
Thus
\begin{align*}
B_{12}P_2{P_2^*} &=P_1B_{12}{P_2^*} =P_1{P_1^*}B_{12}
=B_{12}\,, \\
B_{12}{P_2^*}P_2 &= {P_1^*}B_{12}P_2 ={P_1^*}P_1B_{12}=B_{12}\,,
\end{align*}
and so we have
\[
P_2{P_2^*}B_{12}^*=B_{12}^*={P_2^*}P_2B_{12}^*\,.
\]
This shows that $P_2$ is unitary on the range of $B_{12}^*$ which
can never happen because $P_2$ is completely non-unitary.
Therefore, we must have $B_{12}^*=0$ and so $B_{12}=0$. Similarly
we can prove that $A_{12}=0$. Also from (\ref{eqn:9}), $A_{21}=0$
and from (\ref{eqn:10}), $B_{21}=0$. Thus with respect to the
decomposition $\mathcal H=\mathcal H_1\oplus \mathcal H_2$
\[
A=
\begin{bmatrix}
A_{11}&0\\
0&A_{22}
\end{bmatrix}\,, \quad
B=
\begin{bmatrix}
B_{11}&0\\
0&B_{22}
\end{bmatrix}.
\]
So, $\mathcal H_1$ and $\mathcal H_2$ reduce $A$ and $B$. Also
$(A_{22},B_{22},P_2)$, being the restriction of the $\mathbb
E$-contraction $(A,B,P)$ to the reducing subspace $\mathcal H_2$,
is an $\mathbb E$-contraction. Since $P_2$ is completely
non-unitary, $(A_{22},B_{22},P_2)$ is a completely non-unitary
$\mathbb E$-contraction.

\end{proof}

\section{Operator model}

Wold decomposition breaks an isometry into two parts namely a
unitary and a pure isometry (see Section-I, Ch-1, \cite{nagy}). We
have in Theorem \ref{thm:ti} an analogous decomposition for an
$\mathbb E$-isometry by which an $\mathbb E$-isometry splits into
two parts of which one is an $\mathbb E$-unitary and the other is
a pure $\mathbb E$-isometry. The following theorem gives a
concrete model for pure $\mathbb E$-isometries. Before going to
the theorem, we recall the definition of Toeplitz operator with
operator-valued kernel.\\

For a Hilbert space $E$ let $L^2(E)$ be the space of all
$E$-valued square integrable functions on $\mathbb T$ and let
$H^2(E)$ be the space of analytic elements in $L^2(E)$. Also let
$L^{\infty}(\mathcal B(E))$ denote the space of $\mathcal
B(E)$-valued functions on $\mathbb T$ with finite supremum norm.
For $\phi \in L^{\infty}(\mathcal B(E))$, the Toeplitz operator
$T_{\phi}$ with operator-valued symbol $\phi$ is defined by

\begin{gather*}
T_{\phi}\,:\, H^2(E) \rightarrow H^2(E) \\
T_{\phi}(f)=P(\phi f)
\end{gather*}
where $f\in H^2(E)$ and $P$ is the projection of $L^2(E)$ onto
$H^2(E)$.

\begin{thm}\label{model1}
Let $(\hat{T_1},\hat{T_2},\hat{T_3})$ be a pure $\mathbb
E$-isometry acting on a Hilbert space $\mathcal H$ and let
$A_1,A_2$ denote the fundamental operators of the adjoint
$(\hat{T_1}^*,\hat{T_2}^*,\hat{T_3}^*)$. Then there exists a
unitary $U:\mathcal H \rightarrow H^2(\mathcal D_{{\hat{T_3}}^*})$
such that
\[
\hat{T_1}=U^*T_{\varphi}U,\quad \hat{T_2}=U^*T_{\psi}U \textup{
and } \hat{T_3}=U^*T_zU,
\]
where $\varphi(z)= G_1^*+G_2z,\,\psi(z)= G_2^*+G_1z, \quad
z\in\mathbb T$ and $G_1\,,G_2$ are restrictions of $UA_1U^*$ and
$UA_2U^*$ to the defect space $\mathcal D_{\hat{T_3^*}}$.
Moreover, $A_1,A_2$ satisfy
\begin{enumerate}
\item $[A_1,A_2]=0\,;$ \item $[A_1^*,A_1]=[A_2^*,A_2] \,;$ and
\item $\|A_1^*+A_2z\|\leq 1$ for all $z\in {\mathbb D}$.
\end{enumerate}
Conversely, if $A_1$ and $A_2$ are two bounded operators on a
Hilbert space $E$ satisfying the above three conditions, then
$(T_{A_1^*+A_2z},T_{A_2^*+A_1z},T_z)$ on $H^2(E)$ is a pure
$\mathbb E$-isometry.
\end{thm}

See Theorem 3.3 in \cite{sourav1} for a proof to this theorem. The
following dilation theorem was proved in \cite{tirtha} and for a
proof one can see Theorem 6.1 in \cite{tirtha}.

\begin{thm}\label{thm:dilation}

Let $(A, B, P)$ be a tetrablock contraction on $\mathcal H$ with
fundamental operators $F_1$ and $F_2$ . Let $\mathcal D_P$ be the
closure of the range of $D_P$. Let $\mathcal K = \mathcal H \oplus
\mathcal D_P \oplus \mathcal D_P \oplus \cdots = \mathcal H \oplus
l^2(\mathcal D_P) $. Consider the operators $V_1, V_2$ and $V_3$
defined on $\mathcal{K}$ by
\begin{align*} &
V_1(h_0,h_1,h_2,\dots)=(Ah_0,F_2^* D_P h_0 + F_1 h_1 , F_2^*h_1 + F_1 h_2 , F_2^*h_2 + F_1 h_3,\dots)\\
& V_2(h_0,h_1,h_2,\dots)=(Bh_0 , F_1^* D_P h_0 + F_2 h_1 , F_1^*h_1 + F_2 h_2 , F_1^*h_2 + F_2 h_3,\dots)\\
& V_3(h_0,h_1,h_2,\dots)=(Ph_0, D_P h_0,h_1,h_2,\dots).
\end{align*}
Then \begin{enumerate} \item $\underline V = (V_1,V_2,V_3)$ is a
minimal tetrablock isometric
    dilation of $(A, B, P)$ if $[F_1 , F_2] = 0$ and $[F_1 , F_1^* ] = [F_2 , F_2^* ]$.
\item If there is a tetrablock isometric dilation $\underline W =
    (W_1,W_2,W_3)$ of $(A, B, P)$ such that $W_3$ is the minimal isometric dilation of $P$,
    then $\underline W$ is unitarily equivalent to $\underline V$. Moreover, $[F_1, F_2] = 0$ and $[F_1 , F_1^* ]
    = [F_2 , F_2^* ]$.
\end{enumerate}

\end{thm}

The following result of one variable dilation theory is necessary
for the proof of the model theorem for $\mathbb E$-contractions
and since the result is well-known we do not give a proof here.

\begin{prop}\label{easyprop1}
If $P$ is a contraction and $W$ is its minimal isometric dilation
then $P^*$ and $W^*$ have defect spaces of same dimension.
\end{prop}

The next theorem is the main result of this section and it
provides a model for the $\mathbb E$-contractions which satisfy
some conditions.

\begin{thm}\label{thm:model2}
 Let $(A,B,P)$ be an $\mathbb E$-contraction on a Hilbert space $\mathcal
 H$ and let $F_1,F_2$ and $F_{1*},F_{2*}$ be respectively the
 fundamental operators of $(A,B,P)$ and $(A^*,B^*,P^*)$. Let
 $F_{1*},F_{2*}$ satisfy $[F_{1*},F_{2*}]=0$ and $[F_{1*}^*,F_{1*}]=[F_{2*}^*,F_{2*}]$.
 Let $(T_1,T_2,T_3)$ on $\mathcal K_*=\mathcal H\oplus \mathcal D_{P^*}\oplus\mathcal D_{P^*}\oplus
 \cdots$ be defined as
 \begin{gather*}
 T_1=\begin{bmatrix}
 A&D_{P^*}F_{2*}&0&0&\cdots\\ 0&F_{1*}^*&F_{2*}&0&\cdots\\
 0&0&F_{1*}^*&F_{2*}&\cdots \\ 0&0&0&F_{1*}^*&\cdots\\ \vdots&\vdots&\vdots&\vdots&\ddots
 \end{bmatrix}\,,\quad
 T_2=\begin{bmatrix}
 B&D_{P^*}F_{1*}&0&0&\cdots\\ 0&F_{2*}^*&F_{1*}&0&\cdots\\
 0&0&F_{2*}^*&F_{1*}&\cdots \\ 0&0&0&F_{2*}^*&\cdots\\ \vdots&\vdots&\vdots&\vdots&\ddots
 \end{bmatrix} \,, \\
T_3=\begin{bmatrix} P&D_{P^*}&0&0&\cdots\\0&0&I&0&\cdots\\0&0&0&I&\cdots \\
 0&0&0&0&\cdots\\\vdots&\vdots&\vdots&\vdots&\ddots
 \end{bmatrix}\,.
 \end{gather*}
 Then
 \begin{enumerate}
 \item [(1)] $(T_1,T_2,T_3)$ is an $\mathbb E$-co-isometry, $\mathcal H$ is a
 common invariant subspace of $T_1,T_2,T_3$ and $T_1|_{\mathcal H}=A,\, T_2|_{\mathcal H}=B$ and
 $T_3|_{\mathcal H}=P$; \item [(2)] there is an orthogonal decomposition $\mathcal K_*=\mathcal K_1\oplus \mathcal K_2$
 into reducing subspaces of $T_1,\,T_2$ and $T_3$ such that $(T_1|_{\mathcal K_1},T_2|_{\mathcal K_1},T_3|_{\mathcal K_1})$ is
 an $\mathbb E$-unitary and $(T_1|_{\mathcal K_2},T_2|_{\mathcal K_2},T_3|_{\mathcal K_2})$ is a
 pure $\mathbb E$-co-isometry; \item [(3)] $\mathcal K_2$ can be
 identified with $H^2(\mathcal D_{T_3})$, where $\mathcal D_{T_3}$ has same dimension
 as that of $\mathcal D_P$. The operators $T_1|_{\mathcal K_2},\,T_2|_{\mathcal K_2}$ and $T_3|_{\mathcal
 K_2}$ are respectively unitarily equivalent to $T_{G_1+G_{2}^*\bar z},\,T_{G_{2}+G_{1}^*\bar z}$ and $T_{\bar
 z}$ defined on $H^2(\mathcal D_{T_3})$, $G_1,G_2$ being the fundamental operators of
 $(T_1,T_2,T_3)$.
 \end{enumerate}
 \end{thm}
 \begin{proof}
We apply Theorem \ref{thm:dilation} to $(A^*,B^*,P^*)$ to obtain a
minimal $\mathbb E$-isometric dilation for $(A^*,B^*,P^*)$. If we
denote this $\mathbb E$-isometric dilation by
$(V_{1*},V_{2*},V_{3*})$ then it is evident from Theorem
\ref{thm:dilation} that each $V_{i*}$ is defined on $\mathcal
K_*=\mathcal H\oplus \mathcal D_{P^*}\oplus\mathcal D_{P^*}\oplus
\cdots$ and with respect to this decomposition

\begin{gather*}
V_{1*}= \begin{bmatrix} A^*&0&0&0&\dots\\
F_{2*}^* D_{P^*} & F_{1*}& 0 & 0 &\dots \\
0 & F_{2*}^* & F_{1*} & 0 & \dots \\
0 & 0 & F_{2*}^* & F_{1*} & \dots\\
\dots&\dots&\dots&\dots&\dots\\
\end{bmatrix}\,,
V_{2*}= \begin{bmatrix} B^*&0&0&0&\dots\\
F_{1*}^* D_{P^*} & F_{2*}& 0 & 0 &\dots \\
0 & F_{1*}^* & F_{2*} & 0 & \dots \\
0 & 0 & F_{1*}^* & F_{2*} & \dots\\
\dots&\dots&\dots&\dots&\dots\\
\end{bmatrix}\,,\\
V_{3*}=\begin{bmatrix} P^*&0&0&0&\dots\\
D_{P^*} & 0& 0 & 0 &\dots \\
0 & I & 0 & 0 & \dots \\
0 & 0 & I & 0 & \dots\\
\dots&\dots&\dots&\dots&\dots\\
\end{bmatrix}\,.
\end{gather*}
Obviously $(T_1^*,T_2^*,T_3^*)=(V_{1*},V_{2*},V_{3*})$. It is
clear from the block matrices of $T_i$ that $\mathcal H$ is a
common invariant subspace of each $T_i$ and $T_1|_{\mathcal
H}=A,\, T_2|_{\mathcal H}=B$ and $T_3|_{\mathcal H}=P$. Again
since $(T_1^*,T_2^*,T_3^*)$ is an $\mathbb E$-isometry, by Theorem
\ref{thm:ti}, there is an orthogonal decomposition $\mathcal
K_*=\mathcal K_1\oplus\mathcal K_2$ into reducing subspaces of
$T_i$ such that $(T_1|_{\mathcal K_1},T_2|_{\mathcal
K_1},T_3|_{\mathcal K_1})$ is an $\mathbb E$-unitary and
$(T_1|_{\mathcal K_2},T_2|_{\mathcal K_2},T_3|_{\mathcal K_2})$ is
a pure $\mathbb E$-co-isometry.

If we denote $(T_1|_{\mathcal K_1},T_2|_{\mathcal
K_1},T_3|_{\mathcal K_1})$ by $(T_{11},T_{12},T_{13})$ and
$(T_1|_{\mathcal K_2},T_2|_{\mathcal K_2},T_3|_{\mathcal K_2})$ by
$(T_{21},T_{22},T_{23})$, then with respect to the orthogonal
decomposition $\mathcal K_*=\mathcal K_1\oplus \mathcal K_2$ we
have that
 \[
 T_1=\begin{bmatrix}T_{11}&0\\0&T_{21} \end{bmatrix}\,,\,T_2=\begin{bmatrix}T_{12}&0\\0&T_{22} \end{bmatrix}\,,
 T_3=\begin{bmatrix}T_{13}&0\\0&T_{23} \end{bmatrix}.
 \]
The fundamental equations $T_1-T_2^*T_3=D_{T_3}X_1D_{T_3}$ and
$T_2-T_1^*T_3=D_{T_3}X_2D_{T_3}$ clearly become
\begin{align*}
\begin{bmatrix}T_{11}-T_{12}^*T_{13} &0\\0&T_{21}-T_{22}^*T_{23} \end{bmatrix}
=\begin{bmatrix}0&0\\0& D_{T_{23}}X_{12}D_{T_{23}}
\end{bmatrix},\; X_1=\begin{bmatrix}
X_{11}\\X_{12}\end{bmatrix}
\end{align*}
and
\begin{align*}
\begin{bmatrix}T_{12}-T_{11}^*T_{13} &0\\0&T_{22}-T_{21}^*T_{23} \end{bmatrix}
=\begin{bmatrix}0&0\\0& D_{T_{23}}X_{22}D_{T_{23}}
\end{bmatrix},\; X_2=\begin{bmatrix}
X_{21}\\X_{22}\end{bmatrix}.
\end{align*}
Thus $T_3$ and $T_{23}$ have same defect spaces, that is $\mathcal
D_{T_3}$ and $\mathcal D_{T_{23}}$ are same and consequently
$(T_1,T_2,T_3)$ and $(T_{21},T_{22},T_{23})$ have the same
fundamental operators. Now we apply Theorem \ref{model1} to the
pure $\mathbb E$-isometry
$(T_{21}^*,T_{22}^*,T_{23}^*)=(T_1^*|_{\mathcal
K_2},T_2^*|_{\mathcal K_2},T_3^*|_{\mathcal K_2})$ and get the
following:
\begin{enumerate}
\item[(i)] $\mathcal K_2$ can be identified with $H^2(\mathcal
D_{T_{23}}) (= H^2(\mathcal D_{T_3})$); \item[(ii)]
$(T_{21}^*,T_{22}^*,T_{23}^*)$ can be identified with the
 commuting triple of Toeplitz operators $(T_{G_{1}^*+G_{2}z},T_{G_{2}^*+G_{1} z},T_{z})$
 defined on $H^2(\mathcal D_{T_3})$, where $G_1,G_2$ are the
 fundamental operators of $(T_1,T_2,T_3)$.
\end{enumerate}
 Therefore, $T_1|_{\mathcal K_2}\,, T_2|_{\mathcal K_2}$ and $T_3|_{\mathcal
 K_2}$ are respectively unitarily equivalent to $T_{G_{1}+G_{2}^*\bar z},\,T_{G_{2}+G_{1}^*\bar z}$ and $T_{\bar
 z}$ defined on $H^2(\mathcal D_{T_3})$. The fact that $\mathcal D_{T_3}$ and $\mathcal D_P$
 have same dimensions follows from Proposition \ref{easyprop1} as $T_3^*$ is the minimal isometric dilation of $P^*$.
\end{proof}

\begin{rem}
Theorem \ref{thm:model2} is obtained by applying Theorem
\ref{model1} and Theorem \ref{thm:dilation} (which is Theorem 6.1
in \cite{tirtha}). Theorem \ref{model1} has intersection with
Theorem 5.10 in \cite{tirtha}. Theorem 5.10 in \cite{tirtha} gives
the form of a pure $\mathbb E$-isometry stated in Theorem
\ref{model1}. In Theorem \ref{model1} it has been shown that the
operator-valued kernels $\tau_1,\tau_2$ associated with the
Topelitz operators occurring in Theorem 5.10 of \cite{tirtha} can
be identified with the fundamental operators of the adjoint of the
mentioned pure $\mathbb E$-isometry.
\end{rem}

\noindent \textbf{Acknowledgement.} The author  greatly
appreciates the warm and generous hospitality provided by Indian
Statistical Institute, Delhi during the course of the work.

\end{document}